\documentclass{article}
\usepackage[utf8]{inputenc}
\usepackage[dvips]{graphicx}
 \usepackage{amssymb,amsmath}
 \usepackage{amsthm}
\usepackage{epic}
\usepackage{pstricks}
\usepackage{pst-node}
\usepackage{color}

\usepackage{tikz}
\usetikzlibrary{arrows,decorations.markings}
\usepackage[utf8]{inputenc}
\usepackage{pstricks-add}

\newtheorem{theorem}{Theorem}

\newtheorem{lemma}[theorem]{Lemma}
\newtheorem{corollary}[theorem]{Corollary}
\newtheorem{definition}[theorem]{Definition}

\begin{document}

\title{Bipartite Biregular Cages and Block Designs\footnote{
Research supported by CONACyT-M{\' e}xico under Project 282280 and PAPIIT-M{\' e}xico under Projects IN107218, IN106318, Slovenian Research Agency (research program P1-0285 and research project J1-1695) and VEGA 1/0596/17, VEGA 1/0719/18, APVV-15-0220, by the Slovenian Research Agency (research projects N1-0038, N1-0062, J1-9108), and by NSFC 11371307.}}
\author{Gabriela Araujo-Pardo \\
Instituto de Matematicas, \\
Universidad Nacional Aut\'onoma de M\'exico, M\'exico.
\\
\\
Alejandra Ramos-Rivera\\
University of Primorska\\
Koper, Slovenia.
\\
\\
Robert Jajcay \\
Comenius University, Bratislava, Slovakia, and \\
University of Primorska, Koper Slovenia.}

\maketitle

\begin{abstract}
A {\em bipartite biregular} $(n,m;g)$-graph $G$ is a bipartite graph
of even girth $g$ having the degree set $\{n,m\}$ and satisfying the additional property that 
the vertices in the same partite set have the same degree. 
An  $(n,m;g)$-{\em bipartite biregular cage} is a bipartite biregular $(n,m;g)$-graph 
of minimum order. In their 2019 paper, Filipovski, Ramos-Rivera and Jajcay present lower bounds 
on the orders of bipartite biregular $(n,m;g)$-graphs, and call the graphs that attain these bounds  {\em bipartite biregular Moore cages}. 

In parallel with the well-known classical results relating the existence of $k$-regular Moore graphs of even girths $g = 6,8 $ and $12$ to the existence of projective planes, generalized quadrangles, and generalized hexagons,
we prove that the existence of $S(2,k,v)$-Steiner systems  yields the existence of  bipartite biregular 
$(k,\frac{v-1}{k-1};6)$-Moore cages. Moreover, in the special case of Steiner triple systems (i.e., in the case $k=3$),
we completely solve the problem of the existence of $(3,m;6)$-bipartite biregular cages for all integers $m\geq 4$. 

Considering girths higher than $6$ and prime powers $s$, we relate the existence of
generalized polygons (quadrangles, hexagons and octagons) with the existence of 
$(n+1,n^2+1;8)$, $(n+1,n^3+1;12)$, and $(n+1,n^2+1;16)$-bipartite biregular Moore cages, respectively. Using this connection, we derive improved upper bounds for the orders of bipartite biregular cages of girths $8$, $12$ and $14$.

\end{abstract}




\section{Introduction}
The well-known {\em Cage Problem} is the problem of finding $k$-regular graphs of girth $g$
(which we shall call $(k,g)${\em -graphs})
and smallest possible orders for all pairs of parameters $ k \geq 2, g \geq 3 $ (with the smallest
graphs called $(k,g)${\em -cages}). An exact solution 
for the Cage Problem is only known for very limited sets of parameter pairs $(k,g)$, in particular,
for the cases when the orders of the $(k,g)$-cages match a simple lower bound
due to Moore \cite{ExooJaj08}. These special graphs are called {\em Moore graphs}, and in the case
of even $g=6,8$ or $12$, Moore graphs are known to be the point-block incidence graphs of projective planes, generalized hexagons, and generalized quadrangles, respectively. 

Due to the universally acknowledged complexity of the Cage Problem - and also to 
gain additional insights - several relaxations of the original problem have been introduced
\cite{BobJajPis,ExooJaj08}. One
such relaxation focuses on biregular graphs of girth $g$ and calls for finding smallest graphs
of girth $g$ admitting a two-element degree set $\{ n,m \} $ 
\cite{AraExoJaj,ExooJaj16,FuLaSeUsWol95}. Naturally,
considering graphs from these more relaxed classes requires correspondingly adjusting the 
Moore bounds. Unlike the case of the classical Cage Problem, biregular graphs of degrees 
$n,m$ and girth $g$ whose orders match the adjusted Moore bound were shown to exist for
all pairs $n,m$ where $m$ is considerably larger than $n$ and $g$ is odd \cite{ExooJaj16}.
However, the case of even girth remains wide open \cite{AraExoJaj,YuLi03}. A well-known 
conjecture concerning $(k,g)$-cages of even girth asserts that all even girth cages ought to be
bipartite (see, for example, \cite{Fil17,JaFiJa16}). In connection to this conjecture, 
a further specialization of the biregular cage problem has been introduced in \cite{FilRamRivJaj19}, 
in which the authors proposed to study the {\em {bipartite biregular}} $(n,m;g)$-graphs which are
bipartite graphs of even girth $g$ having the degree set $\{n,m\}$ and satisfying the additional property that the vertices in the same partite set have the same degree. This is also the class we 
consider in our paper. The corresponding {\em adjusted Moore lower bounds} $B(n,m;2r)$ for the orders of the bipartite biregular $(n,m;2r)$-cages derived in \cite{FilRamRivJaj19} take the form:
\begin{eqnarray*}\label{Moore:bound-even}
1  + & (m-1)+(m-1)(n-1)+\ldots+(m-1)^{\frac{r}{2}}(n-1)^{\frac{r}{2}-1} +\\
1 + & (n-1)+(n-1)(m-1)+\ldots+(n-1)^{\frac{r}{2}}(m-1)^{\frac{r}{2}-1},
\end{eqnarray*}
when $r$ is even, and 
\begin{eqnarray*}\label{Moore:bound-odd}
1+ & (m-1)+(m-1)(n-1)+\ldots+(m-1)^{\frac{r-1}{2}}(n-1)^{\frac{r-1}{2}} + \\
1+ & (n-1)+(n-1)(m-1)+\ldots+(n-1)^{\frac{r-1}{2}}(m-1)^{\frac{r-1}{2}},
\end{eqnarray*}
when $r$ is odd. 

The difference between the order of a bipartite biregular $(n,m;2r)$-graph $G$ and 
the corresponding adjusted Moore bound $B(m,n;2r)$ is called the {\em excess} of $G$,
and in \cite{FilRamRivJaj19}, the authors succeeded at proving the non-existence of bipartite 
$(m,n;g)$-graphs of excess at most $4$ for all parameters $m, n, g$ where $g\geq10$ and is 
not divisible by $4$, and $m > n\geq3$. They also proved for all pairs $m,n$, $m>n\geq 3$, that the asymptotic density of the set of even girths $g\geq8$ for which there exists a bipartite $(m,n;g)$-graph with excess not exceeding $4$ is equal to $0$. Following
the usual way of naming graphs whose orders match the Moore bound, we shall call bipartite biregular $(n,m;2r)$-graphs whose orders match the corresponding adjusted Moore bound (\ref{Moore:bound-even}) or 
(\ref{Moore:bound-odd}) {\em  bipartite biregular Moore $(n,m;2r)$-cages}.

In this paper, we use a a connection to balanced incomplete block designs to present infinite families of bipartite biregular Moore $(n,m;6)$-cages; and, for larger girths, 
we explore a connection between bipartite biregular Moore $(n,m;g)$-graphs, 
$ g = 8,12 $ or $16$, and generalized polygons. Our approach is analogous to a well
established connection between the existence of $(k,g)$-Moore graphs and finite geometries.

\section{Preliminaries}
In this section we introduce definitions and theorems concerning block designs and generalized polygons which are well-known to experts in these fields. Nevertheless,
for the sake of completeness, we include these as they will be used throughout our paper.

\subsection{Block designs and Steiner systems}\label{BDSS}
\label{sec:pre}
All definitions and results given in this subsection can be found for example in \cite{ColDin}.

\begin{definition}
A {\em balanced incomplete block design} (BIBD) is a collection of $k$-subsets (called {\em blocks}) of a $v$-set $S$, $k<v$, such that each pair of elements of $S$ occur together in exactly $\lambda$  blocks. 
\end{definition}
The number of blocks is denoted by $b$ and the design is denoted as a $(v,k,\lambda)$-{\em design}.

\begin{theorem}
\label{eq}
In a $(v,k,\lambda)$-design with $b$ blocks each element occurs in $r$ blocks where:
\begin{enumerate}
\item $r(k-1)=\lambda(v-1)$, and
\item $vr=bk$.
\end{enumerate}
\end{theorem}
Sometimes a $(v,k,\lambda)$-design is described as a $(v,b,r,k,\lambda)$-design, but it is important to realize that, once $v$, $k$ and $\lambda$ are known, the other two parameters are uniquely determined. 
The next concept is a generalization of a BIBD.

\begin{definition}
Let $t,\ v,\ k,\ \lambda$ be integers with $v\geq k\geq t \geq 2$ and $\lambda > 0$. A $t-(v,k,\lambda)$-{\em design} is a collection of $k$-subsets ({\em blocks}) of a $v$-set $S$ such that every $t$-subset of $S$ is contained in exactly $\lambda$ blocks. We observe that a $2$-design is just a BIBD. 
\end{definition}

A Steiner system is a $t$-design with $\lambda=1$, that is: 

\begin{definition}
A {\em{Steiner system}} $S(t,k,v)$ is a collection of $k$-subsets (blocks) of a $v$-subset $S$ such that every $t$-subset of $S$ is contained in exactly one of the blocks. Note that we use $S(t,k,v)$ as an equivalent to $t-(k,v,1)$. Note also that Steiner systems $S(2,k,v)$ are just $(v,k,1)$ BIBDs. 
\end{definition}

In this paper, we will be specifically interested in Steiner systems, for which we will show that their
point-block incidence graphs are related to  bipartite biregular cages; the main topic of our paper. 
In particular, we obtain our best results when $k=3$, in which case the Steiner systems $S(2,3,v)$  are called {\em Steiner triple systems} and are denoted by $STS(v)$.\\

The equalities given in Theorem \ref{eq} yield that if a Steiner system exists, it necessarily holds that $v\equiv 1  \pmod{k(k-1)}$ or $v\equiv k  \pmod{k(k-1)}$. 

Steiner systems are known to exist for certain specific values of $k$ and $v$, but the general question for 
which values of $k$ there exists a Steiner system is a well-known open problem; the congruences given above constitute necessary but not sufficient conditions. However, for $k=3$, it has already been proved by 
Kirkman that these conditions are both necessary and sufficient, and an $STS(v)$ exists if and only 
$ v \geq 7 $ is congruent to $1$ or $3$ modulo $6$ \cite{ColDin}.

\subsection{Generalized Polygons}\label{QHG}
To learn more about the definitions and results given in this subsection consult for
example \cite{VM98}. A {\em geometry} $ (\cal{P},\cal{L},{\bf I}) $ is a triple consisting
of disjoint sets $ \cal{P} $  and $ \cal{L} $, called {\em points} and {\em lines} related through the {\em incidence relation} $ {\bf I} \subseteq \cal{P} \times \cal{L} $ 
that determines which points from $ \cal{P} $ belong to which lines in $ \cal{L} $.
An {\em ordinary} $n${\em -gon} arises naturally from the regular $n$-gon in the 
plane. 

\begin{definition}
Let $n\geq 1$ be a natural number. A {\em weak generalized $n$-gon} is a geometry 
$\Gamma = (\cal{P},\cal{L},{\bf I})$ such that the following two axioms are satisfied: 
\begin{enumerate}
\item $\Gamma$ contains no ordinary $k$-gon (as a subgeometry) for $2\leq k < n.$
\item Any two elements $x,y \in \cal{P}\cup\cal{L}$ are contained in some ordinary $n$-gon (again as a subgeometry) in $\Gamma$. 
\vskip 2mm

\hskip -4.5mm
A {\em generalized $n$-gon} is a weak generalized $n$-gon $\Gamma$ which satisfies the additional axiom: 
\item There exists an ordinary $(n+1)$-gon (as a subgeometry) in $\Gamma$. 
\end{enumerate}
\end{definition}

As with the Steiner systems, we will be interested in the point-line incidence graphs
of (weak) generalized $n$-gons. The following lemma establishes the fundamental 
connection.

\begin{lemma}[\cite{VM98}, Lemma 1.3.6]\label{VMLemma}
A geometry $\Gamma = (\cal{P},\cal{L},{\bf I})$ is a (weak) generalized $n$-gon if and
only if the incidence graph of $\Gamma$ is a connected graph of diameter $n$ and
girth $2n$, such that each vertex lies on at least three (at least two) edges.
\end{lemma}

Every generalized polygon $\Gamma$ can be associated with a pair $(s,t)$, 
called the {\em order} of $\Gamma$, such that every line is incident with $s+1$
points, and every point is incident with $t+1$ lines \cite{VM98}. This means, 
in particular, that the incidence graph of $\Gamma$ is a bipartite biregular graph
with the two degrees $s+1$ and $t+1$. In addition, the following result 
determines the orders of the two partite sets.

\begin{theorem}[\cite{VM98}, Corollary 1.5.5]\label{PL}
If there exists a $\Gamma = (\cal{P},\cal{L},{\bf I})$ (weak) generalized $n$-gon of order $(s,t)$ for $n \in \{3,4,6,8\}$ then:
\[ | {\cal P} | = \left\{ \begin{array}{cr} s^2+s+1 & \mbox{ if } n=3; \\
(1+s)(1+st), & \mbox{ if } n=4; \\
(1+s)(1+st+s^2t^2), & \mbox{ if } n=6; \\ (1+s)(1+st)(1+s^2t^2), & \mbox{ if } n=8.
\end{array} \right. \]
\[ |{\cal L} | = \left\{ \begin{array}{cr} t^2+t+1 & \mbox{ if } n=3; \\
(1+t)(1+st), & \mbox{ if } n=4; \\
(1+t)(1+st+s^2t^2), & \mbox{ if } n=6; \\ (1+t)(1+st)(1+s^2t^2), & \mbox{ if } n=8.
\end{array} \right. \]
\end{theorem}
In 1964, Feit and Higman proved that finite generalized $n$-gons exist only for $n=\{3,4,6,8\}$. When $n=3$, we have the projective planes; when $r=4$, we have the generalized quadrangles, which are known to exist when $s=t$ or $t=s^2$; when $r=6$, we have the generalized hexagons, which also exist in the regular case $s=t$ 
and when $t=s^3$; and finally, for $r=8$, we have the generalized octagons, which only exist when $s=t^2$. 

The rest of the paper is dedicated to our contributions.
\section {Bipartite biregular cages of girth 6}
\label{sec:girth6}

As argued in \cite{FilRamRivJaj19,YuLi03}, any $(n,m;6)$-graph has at least the following number of vertices 
\begin{equation}
\label{eq:1}
n_0(n,m;6)=\left\lceil n+m+2(m-1)(n-1)+ \frac{(m-1)(n-1)(m-n)}{n}\right\rceil =\left\lceil \frac{(m+n)(mn-m+1)}{n}\right\rceil,
\end{equation}
and we shall call the graphs whose orders attain this lower bound the {\em $(n,m;6)$-Moore cages}. In general, given an $(n,m;6)$-graph $G$, {\em the excess} of $G$ is 
the difference $e=|V(G)|-n_0(n,m;6)$.

\subsection{Bipartite biregular Moore cages associated to Steiner systems}
In this section we will prove that the existence of a Steiner system immediately yields
the existence of a bipartite biregular Moore cage.

\begin{theorem}
\label{the: S(2,k,v)}
The point-block incidence graph $G$ of a Steiner system $S(2,k,v)$, $ k \geq 3 $, is a 
(bipartite biregular) $(k,r;6)$-Moore cage, where $r=\frac{v-1}{k-1}$. 
\end{theorem}
\begin{proof}
The fact that the point-block incidence graph $G$ of a Steiner system $S(2,k,v)$,
$ k \geq 3 $, is a $(k,r;6)$-graph is fairly obvious: Each point belongs to $r$ blocks,
each block contains $k$ vertices, $G$ is bipartite and therefore contains even length
cycles only, and the existence of a $4$-cycle in $G$ would force the two vertices of
the $4$-cycle to belong to both of its blocks, which would violate the condition $ \lambda 
=1 $. Since every pair of vertices belongs to a block, any $3$-set of vertices 
not belonging to the same block of the Steiner system together with the three blocks containing the three $2$-subsets of this $3$-set forms a $6$-cycle.

Thus, to complete the proof of our theorem, we only need to show that the order of the
point-block incidence graph $G$ of a Steiner system $S(2,k,v)$ attains the lower 
bound (\ref{eq:1}). Using Theorem~\ref{eq} implies $v=r(k-1)+1$ and $r=bk/v$, and 
therefore
$$ | V(G) | = v+b=(r(k-1)+1) + \left(\frac{r^2(k-1)+r}{k}\right)$$
$$=\frac{(r+k)(rk-r+1)}{k} =\left\lceil \frac{(r+k)(rk-r+1)}{k} \right\rceil.$$ 
Hence, $ | V(G) | = n_0(r,k;6) $, as defined in (\ref{eq:1}), and the incidence graph of 
the Steiner system $S(2,k,v)$ is a $(k,r;6)$-Moore cage. 
\end{proof}

Obviously, the above graphs are truly biregular only if $ r \neq k $. Recall that if 
the Steiner system $S(2,k,v)$ is in fact a projective plane of order $r-1$, then $r=k$.
Thus, Theorem~\ref{the: S(2,k,v)} is a generalization of the well-known result 
stating that
the point-line incidence graph of a projective plane of order $q$ is a $(q+1,6)$-Moore
cage (see, e.g., \cite{ExooJaj08}).
Moreover, as the affine plane of order $q$ is also a Steiner system $S(2,q,q^2)$, 
we have the following result:  

\begin{corollary}
\label{cor:affine planes}
The point-line incidence graph $G$ of an affine plane of order $q$ is a 
(bipartite biregular) $(q,q+1;6)$-Moore cage. 
\end{corollary}
In particular, Figure 5 in \cite{FilRamRivJaj19} is the incidence graph of a affine plane of order $3$. 

\subsection{Bipartite biregular  cages associated with Steiner triple system.}
The results of the previous subsection show that the existence of a Steiner system
$S(2,k,v)$ gives rise to a bipartite biregular $(k,r;6)$-Moore cage, where the parameter
$r$ is determined by the parameters $k$ and $v$. Since no classification of the 
parameter pairs $(k,v)$ admitting the existence of an $S(2,k,v)$ exists, this type of 
result does not allow for classifying the parameter pairs $(k,r)$ admitting the existence
of a bipartite biregular $(k,r;6)$-Moore cage.

As pointed out in Section~\ref{BDSS}, the $k=3$ situation is different in that we have
a classification of the pairs $(3,v)$ that admit the existence of an $S(2,3,v)$.
Thus, in this section we will focus on Steiner triple systems and determine the orders
of the $(3,m;6)$-bipartite biregular cages for all values of $m$.
 
Recall that a Steiner triple system $STS(v)$ exists if and only if $ v \geq 7 $ and 
$v\equiv 1,3 \ (\mod{6})$; in which case $ r = \frac{v-1}{k-1} = \frac{1}{2}(v-1) $.
Thus, Theorem~\ref{the: S(2,k,v)} implies the existence of a bipartite biregular 
$(k,r;6)$-Moore cage for all $r$ of the form $ \frac{1}{2}(6m+1-1)=3m$ or $
 \frac{1}{2}(6m+3-1)=3m+1$, $ m \geq 1 $. Hence, we know the orders of 
 $(3,m;6)$-bipartite biregular cages for two thirds of parameters $m$, $ m \equiv 0 $
 and $ m \equiv 1 \pmod{3} $. In the rest
 of this section, we will focus on the remaining unresolved case $ m \equiv 2 \pmod{3} $.
 
\begin{lemma}
\label{le:5}
Let $m>4$ be an integer. If $m\equiv 2 \pmod{3}$, then the order $|V(G)|$ of any 
$(3,m;6)$-graph $G$ satisfies the inequality:
$$\frac{2m^2+7m+3}{3}+\frac{m+3}{3} \leq  |V(G)|.$$
\end{lemma}
\begin{proof}
Let $m>4$, $m\equiv 2 \pmod{3}$, and let $G$ be a $(3,m;6)$-bipartite biregular graph with
the sets $A$ and $B$ forming the partition of $V(G)$; with $A$ of cardinality $a$ 
containing the vertices of degree $m$, and $B$ of cardinality $b$ containing the vertices of degree $3$. The order of $G$ is then equal to $a+b$, and since the number of edges 
leaving $A$ must be equal to the number of edges entering $B$, $am=3b$.
The lower bound from formula (\ref{eq:1}) takes the form
\[ n_0(3,m;6)=\left\lceil \frac{(m+3)(3m-m+1)}{3}\right\rceil = \left\lceil  \frac{2m^2+7m+3}{3}   \right\rceil , \]
and since $ m \equiv 2 \pmod{3} $,
\[ n_0(3,m;6)= \left\lceil  \frac{2m^2+7m+3}{3}   \right\rceil = \frac{2m^2+7m+3}{3} + \frac{2}{3} . \]

\noindent
Let $e$ denote the excess of $G$ (as compared to $n_0(3,m;6)$), and recall that $b=\frac{am}{3}$. Then,
$$ |V(G) | =\frac{2m^2+7m+3}{3} + \frac{2}{3} + e = \frac{2m^2+7m+3+2+3e}{3} , $$
while also $ |V(G) | = a+\frac{am}{3} $,
and therefore 
$$a(3+m)=2m^2+7m+3+2+3e . $$
In particular, $a=\frac{2m^2+7m+3+2+3e}{m+3}$ must be an integer.
Since $2m^2+7m+3 = (m+3)(2m+1)$, the expression $m+3$ divides $2m^2+7m+3$ evenly, and therefore 
the remaining fraction $\frac{2+3e}{m+3}$ must be an integer
as well. However, if $e$ were smaller than $\frac{m+1}{3}$, we would have $\frac{2+3e}{m+3} < 1$, and 
therefore, $ \frac{m+1}{3}  \leq e$. 

Thus, 
\[ |V(G)| = \frac{2m^2+7m+3+2}{3}+e \geq \frac{2m^2+7m+3}{3}+\frac{m+3}{3} . \]
 
\end{proof}

\begin{theorem}
\label{le:STS}
Let $m>3$ be an integer. 
\begin{enumerate}
\item If $m \equiv 0,1 \pmod{3}$, then a $(3,m;6)$-biregular bipartite cage is the point-block incidence graph of 
the Steiner triple system $STS(2m+1)$ of order $n_0(3,m;6)$. 
\item If $m \equiv 2 \pmod{3}$, then the order of a $(3,m;6)$-biregular bipartite cage is 
$n_0(3,m;6)+\frac{m+3}{3}$. 
\end{enumerate}
\end{theorem}

\begin{proof} 
If $m \equiv 0,1 \pmod{3}$, then $2m+1 \equiv 1 $ or $3 \pmod{6}$, the Steiner triple system 
$STS(2m+1)$ exists, and its point-block incidence graph of order $n_0(3,m;6)$
is a $(3,m;6)$-biregular bipartite cage as argued in Theorem~\ref{the: S(2,k,v)}.

Next assume that $m \equiv 2 \pmod{3}$. Lemma~\ref{le:5} asserts that the order of any 
$(3,m;6)$-biregular bipartite graph $G$ is greater than or equal to $n_0(3,m;6)+\frac{m+3}{3}$. 
Thus, to prove the second claim of our theorem, we just need to prove the existence of a
$(3,m;6)$-biregular bipartite graph $G$ of order equal to $n_0(3,m;6)+\frac{m+3}{3}$.

Since $m \equiv 2 \pmod{3}$, it follows that $2m+3 \equiv 1 \ (\mod{6})$, and thus an $STS(2m+3)$
exists. Let $G'$ be the $(3,m+1;6)$-biregular bipartite incidence graph of this Steiner triple system.
We shall construct the desired $(3,m;6)$-biregular bipartite graph $G$ of order $n_0(3,m;6)+\frac{m+3}{3}$
by removing vertices and edges from the graph $G'$.

Let $x$ be a point and $b$ be a block of the $STS(2m+3)$ containing $x$. Let $f=(x,b)$ be an edge of 
$G'$, and let $G=G'[X]$ be the subgraph of $G'$ induced by the vertex set $X=V(G')\backslash (x\cup N(x))$,
where $N(x)$ is the neighborhood of $x$ in $G'$, i.e., the set of all blocks of the $STS(2m+3)$ containing
$x$. We claim that $G$ is a $(3,m;6)$-biregular bipartite graph (see Figure \ref{fig:G}). 
First, none of the remaining blocks contains the 
point $x$, and hence each of them is connected to $3$ of the remaining points. Hence all the remaining
blocks represent vertices of degree $3$ and no two of them are adjacent. As for the remaining points, they
all have been originally adjacent to $m+1$ blocks, and are now adjacent to one less block - the block they
shared with $x$ - and hence represent mutually non-adjacent vertices of degree $m$. Since $G$ contains
no vertices of degree $1$ or $0$, it is not a forest and contains at least one cycle. We claim that the number
of vertices of $G$ is too small for $G$ to be of girth $8$. To prove this, as well as to prove that $G$ is a 
$(3,m;6)$-cage, we determine the order of $G$:
$$
\begin{array}{ccccccccc}
|V(G)|&=&|V(G')|-1-(m+1)\\
&=&\frac{2(m+1)^2+7(m+1)+3}{3}-2-m\\ 
&=&\frac{2m^2+4m+2+7m+7+3-6-3m}{3}\\
&=&\frac{2m^2+8m+6}{3} \\
&=&\frac{2m^2+7m+3}{3}+\frac{m+3}{3}\\
&=&n_0(3,m;6)+\frac{m+3}{3}.\\

\end{array}
$$
Thus, $G$ has the desired order. As proved in \cite{FilRamRivJaj19}, the order of any $(3,m;8)$-biregular
bipartite graph is bounded from below by 
\[ 1 + 2 + 2(m-1) + 4(m-1) + 1 + (m-1) + 2(m-1) + 2(m-1)^2 = -3 + 5m + 2m^2, \]
a value that is clearly bigger than $ |V(G)| = \frac{2m^2+8m+6}{3} $, and thus the girth of $G$ cannot 
be $8$.

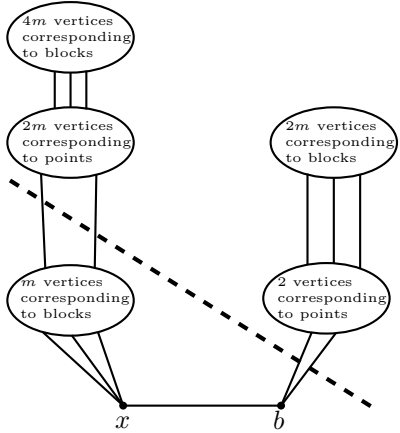
\begin{figure}[h!]
\begin{center}
\psset{xunit=.7cm,yunit=.7cm,algebraic=true,dimen=middle,dotstyle=o,dotsize=3pt 0,linewidth=0.8pt,arrowsize=3pt 2,arrowinset=0.25}
\begin{pspicture*}(5.2,1.45)(13.68,9.8)
\rput{0}(7,4){\psellipse(0,0)(1.2,0.67)}
\rput{0}(11.86,4.04){\psellipse(0,0)(1.2,0.67)}
\rput{0}(7,7){\psellipse(0,0)(1.2,0.67)}
\rput{0}(12,7){\psellipse(0,0)(1.2,0.67)}
\rput{0}(7,9){\psellipse(0,0)(1.2,0.67)}
\psline(8,2)(11,2)

\psline(8,2)(7.02,3.34)
\psline(8,2)(7.52,3.4)
\psline(8,2)(6.47,3.4)

\psline(11,2)(12.04,3.38)
\psline(11,2)(11.58,3.39)

\psline(12,4.7)(12,6.33)
\psline(12.5,4.6)(12.5,6.39)
\psline(11.5,4.67)(11.5,6.41)

\psline(6.52,4.61)(6.44,6.41)
\psline(7.46,4.62)(7.49,6.39)

\psline(6.7,7.64)(6.7,8.36)
\psline(7,7.67)(7,8.33)
\psline(7.3,7.63)(7.3,8.37)

\rput[lt](10.93,4.44){\parbox{1.6 cm}{\tiny{$2$ vertices  \\ corresponding \\  to points}}}
\rput[lt](6.05,4.44){\parbox{1.6 cm}{\tiny{$m$ vertices  \\ corresponding \\  to blocks}}}
\rput[lt](6.08,7.4){\parbox{1.65 cm}{\tiny{$2m$ vertices  \\ corresponding \\  to points}}}
\rput[lt](11.08,7.4){\parbox{1.65 cm}{\tiny{$2m$ vertices  \\ corresponding \\  to blocks}}}
\rput[lt](6.08,9.41){\parbox{1.65 cm}{\tiny{$4m$ vertices  \\ corresponding \\  to blocks}}}
\rput[tl](7.85,1.8){$x$}
\rput[tl](10.85,1.9){$b$}
\psline[linewidth=1.6pt,linestyle=dashed,dash=4pt 4pt](5.85,6.28)(12.7,1.99)
\begin{scriptsize}
\psdots[dotstyle=*](8,2)
\psdots[dotstyle=*](11,2)
\end{scriptsize}
\end{pspicture*}
\caption{A sketch of how to obtain $G$.}
\label{fig:G}
\end{center}
\end{figure}
\end{proof}

For an example, consider Figure \ref{fig:m5} which depicts a part of the bipartite biregular graph $G'$ that is a $(3,6;6)$-Moore cage. It is the incidence graph of a Steiner triple system $STS(13)$ with the blue points corresponding to the $13$ elements (denoted as $\{0,1,\ldots,9,A,B,C\}$) and the red points corresponding to the $26$ blocks (note that the edges that are not drawn in Figure \ref{fig:m5} can be easily determined by the incidence rule: the point 
$x$ belongs to the blocks $xyz$ independently of the values of $y$ an $z$). 

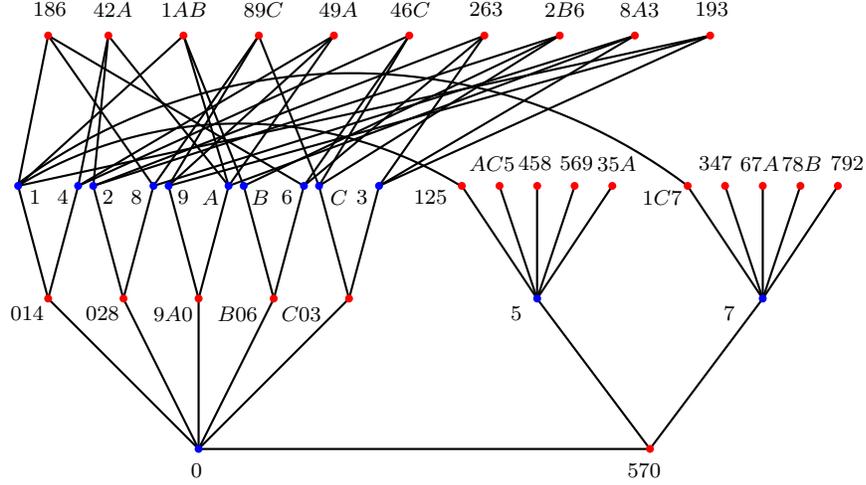
\begin{figure}[htbp]
\begin{center}
\psset{xunit=1.0cm,yunit=1.0cm,algebraic=true,dimen=middle,dotstyle=o,dotsize=3pt 0,linewidth=0.8pt,arrowsize=3pt 2,arrowinset=0.25}
\begin{pspicture*}(-1.2,0.61)(11.95,7.84)
\rput[tl](-0.25,4.95){\footnotesize{$1$}}
\rput[tl](.12,4.95){\footnotesize{$4$}}
\rput[tl](0.72,4.95){\footnotesize{$2$}}
\rput[tl](1.1,4.95){\footnotesize{$8$}}
\rput[tl](1.72,4.95){\footnotesize{$9$}}
\rput[tl](2.05,4.95){\footnotesize{$A$}}
\rput[tl](2.7,4.95){\footnotesize{$B$}}
\rput[tl](3.1,4.95){\footnotesize{$6$}}
\rput[tl](3.75,4.95){\footnotesize{$C$}}
\rput[tl](4.1,4.95){\footnotesize{$3$}}

\rput[tl](-0.2,7.45){\footnotesize{$186$}}
\rput[tl](0.6,7.45){\footnotesize{$42A$}}
\rput[tl](1.5,7.45){\footnotesize{$1AB$}}
\rput[tl](2.6,7.45){\footnotesize{$89C$}}
\rput[tl](3.6,7.45){\footnotesize{$49A$}}
\rput[tl](4.55,7.45){\footnotesize{$46C$}}
\rput[tl](5.6,7.45){\footnotesize{$263$}}
\rput[tl](6.6,7.45){\footnotesize{$2B6$}}
\rput[tl](7.6,7.45){\footnotesize{$8A3$}}
\rput[tl](8.6,7.45){\footnotesize{$193$}}

\rput[tl](-0.5,3.4){\footnotesize{$014$}}
\rput[tl](0.5,3.4){\footnotesize{$028$}}
\rput[tl](1.4,3.4){\footnotesize{$9A0$}}
\rput[tl](2.25,3.4){\footnotesize{$B06$}}
\rput[tl](3.1,3.4){\footnotesize{$C03$}}

\rput[tl](1.9,1.3){\footnotesize{$0$}}
\rput[tl](7.7,1.3){\footnotesize{$570$}}
\rput[tl](6.15,3.4){\footnotesize{$5$}}
\rput[tl](8.98,3.4){\footnotesize{$7$}}

\rput[tl](7.9,4.95){\footnotesize{$1C7$}}
\rput[tl](8.65,5.4){\footnotesize{$347$}}
\rput[tl](9.2,5.4){\footnotesize{$67A$}}
\rput[tl](9.75,5.4){\footnotesize{$78B$}}
\rput[tl](10.4,5.4){\footnotesize{$792$}}

\rput[tl](4.86,4.95){\footnotesize{$125$}}
\rput[tl](5.6,5.4){\footnotesize{$AC5$}}
\rput[tl](7.3,5.4){\footnotesize{$35A$}}
\rput[tl](6.25,5.4){\footnotesize{$458$}}
\rput[tl](6.8,5.4){\footnotesize{$569$}}

\psline(2,1.5)(8,1.5)
\psline(6.5,3.5)(8,1.5)
\psline(8,1.5)(9.5,3.5)
\psline(2,1.5)(0,3.5)
\psline(2,1.5)(3,3.5)
\psline(6.5,3.5)(6,5)
\psline(6.5,3.5)(7,5)
\psline(6.5,3.5)(5.5,5)
\psline(2,1.5)(4,3.5)
\psline(2,1.5)(2,3.5)
\psline(2,1.5)(1,3.5)
\psline(6.5,3.5)(6.5,5)
\psline(6.5,3.5)(7.5,5)
\psline(9.5,3.5)(9.5,5)
\psline(9.5,3.5)(9,5)
\psline(9.5,3.5)(8.5,5)
\psline(9.5,3.5)(10,5)
\psline(9.5,3.5)(10.5,5)
\psline(4,3.5)(4.4,5)
\psline(4,3.5)(3.6,5)
\psline(3,3.5)(3.4,5)
\psline(3,3.5)(2.6,5)
\psline(2,3.5)(2.4,5)
\psline(2,3.5)(1.6,5)
\psline(1,3.5)(1.4,5)
\psline(0.6,5)(1,3.5)
\psline(0,3.5)(0.4,5)
\psline(0,3.5)(-0.4,5)

\psline(-0.4,5)(1.8,7)
\psline(-0.4,5)(0,7)
\psline(0.4,5)(3.8,7)
\psline(0.4,5)(0.8,7)
\psline(0.4,5)(4.8,7)
\psline(0.6,5)(6.8,7)
\psline(0.6,5)(5.8,7)
\psline(0.6,5)(0.8,7)
\psline(1.4,5)(0,7)
\psline(1.4,5)(2.8,7)
\psline(1.4,5)(7.8,7)
\psline(1.6,5)(2.8,7)
\psline(1.6,5)(3.8,7)
\psline(2.4,5)(1.8,7)
\psline(2.4,5)(0.8,7)
\psline(2.4,5)(7.8,7)
\psline(2.6,5)(6.8,7)
\psline(2.4,5)(3.8,7)
\psline(2.6,5)(1.8,7)
\psline(3.4,5)(0,7)
\psline(3.4,5)(4.8,7)
\psline(3.4,5)(5.8,7)
\psline(3.6,5)(2.8,7)
\psline(3.6,5)(6.8,7)
\psline(3.6,5)(4.8,7)
\psline(4.4,5)(7.8,7)
\psline(4.4,5)(5.8,7)

\parametricplot{1.0193196127018804}{2.1204355467391705}{1*5.64*cos(t)+0*5.64*sin(t)+2.54|0*5.64*cos(t)+1*5.64*sin(t)+0.19}
\psline(-0.4,5)(8.8,7)
\psline(8.8,7)(4.4,5)
\psline(8.8,7)(1.6,5)
\parametricplot{0.923831965944893}{2.2128864854712793}{1*7.42*cos(t)+0*7.42*sin(t)+4.03|0*7.42*cos(t)+1*7.42*sin(t)+-0.92}

\begin{scriptsize}
\psdots[dotstyle=*,linecolor=blue](2,1.5)
\psdots[dotstyle=*,linecolor=red](8,1.5)
\psdots[dotstyle=*,linecolor=blue](6.5,3.5)
\psdots[dotstyle=*,linecolor=blue](9.5,3.5)
\psdots[dotstyle=*,linecolor=red](0,3.5)
\psdots[dotstyle=*,linecolor=red](3,3.5)
\psdots[dotstyle=*,linecolor=red](6,5)
\psdots[dotstyle=*,linecolor=red](7,5)
\psdots[dotstyle=*,linecolor=red](5.5,5)
\psdots[dotstyle=*,linecolor=red](4,3.5)
\psdots[dotstyle=*,linecolor=red](2,3.5)
\psdots[dotstyle=*,linecolor=red](1,3.5)
\psdots[dotstyle=*,linecolor=red](6.5,5)
\psdots[dotstyle=*,linecolor=red](7.5,5)
\psdots[dotstyle=*,linecolor=red](9.5,5)
\psdots[dotstyle=*,linecolor=red](9,5)
\psdots[dotstyle=*,linecolor=red](8.5,5)
\psdots[dotstyle=*,linecolor=red](10,5)
\psdots[dotstyle=*,linecolor=red](10.5,5)
\psdots[dotstyle=*,linecolor=blue](4.4,5)
\psdots[dotstyle=*,linecolor=blue](3.6,5)
\psdots[dotstyle=*,linecolor=blue](3.4,5)
\psdots[dotstyle=*,linecolor=blue](2.6,5)
\psdots[dotstyle=*,linecolor=blue](2.4,5)
\psdots[dotstyle=*,linecolor=blue](1.6,5)
\psdots[dotstyle=*,linecolor=blue](1.4,5)
\psdots[dotstyle=*,linecolor=blue](0.6,5)
\psdots[dotstyle=*,linecolor=blue](0.4,5)
\psdots[dotstyle=*,linecolor=blue](-0.4,5)
\psdots[dotstyle=*,linecolor=red](7.8,7)
\psdots[dotstyle=*,linecolor=red](1.8,7)
\psdots[dotstyle=*,linecolor=red](3.8,7)
\psdots[dotstyle=*,linecolor=red](4.8,7)
\psdots[dotstyle=*,linecolor=red](0,7)
\psdots[dotstyle=*,linecolor=red](0.8,7)
\psdots[dotstyle=*,linecolor=red](6.8,7)
\psdots[dotstyle=*,linecolor=red](5.8,7)
\psdots[dotstyle=*,linecolor=red](2.8,7)
\psdots[dotstyle=*,linecolor=red](8.8,7)

\end{scriptsize}
\end{pspicture*}
\caption{Part of the (6,3;6)-cage.}
\label{fig:m5}
\end{center}
\end{figure}

On one hand, 
$$n_0(3,5;6)=\left\lceil \frac{2(5)^2+7(5)+3}{3}\right\rceil = 30 \mbox{ and } \left\lceil \frac{5}{3} + \frac{1}{3} \right\rceil =2 \leq e, $$
hence $$32 \leq |V(G)|.$$

On the other hand, it follows from Figure~\ref{fig:m5} that 
$$|V(G)|=|V(G')|-2-5=\frac{2(6)^2+7(6)+3}{3}-7=\frac{2(36)+42+3}{3}-7=\frac{117}{3}-7$$
$$|V(G)|=39-7=32,$$
and $G$ is necessarily a cage.

%
%

\section {Bipartite biregular  cages of girths 8, 12 and 16}

We close our paper with two results concerning biregular bipartite $(n,m;2r)$-graphs for $r\in\{4,6,12\}$ using the identities given in Theorem \ref{PL}. The first result follows immediately from Lemma~\ref{VMLemma}, Theorem \ref{PL} and the existence results mentioned in the paragraph following these.
It deals with biregular bipartite Moore cages of girths $8,12$ and $16$.
\begin{theorem}\label{bb-cages}
There exist infinite families of biregular bipartite Moore cages with parameters
$(n+1,n^2+1;8)$, $(n+1,n^3+1;12)$ and $(n+1,n^2+1;16)$.
\end{theorem}

Finally, we address the existence of biregular bipartite $(n,m;2r)$-graphs in relation 
to the Moore  bound $B(n,m;2r)$ (\ref{Moore:bound-even}) and (\ref{Moore:bound-odd}) proved in 
\cite{FilRamRivJaj19}. The idea behind the proof included in \cite{FilRamRivJaj19}
is to consider a tree denoted by ${\cal{T}}_{uv}$ whose order matches the Moore 
bound $B(n,m;2r)$ and is called the {\em{Moore tree}}. 
It is necessarily contained in any biregular bipartite $(n,m;2r)$-graph
and consists of trees ${\cal{T}}_u$ and ${\cal{T}}_v$ attached to the vertices $u$ 
and $v$ of an edge $f=uv$. Each of the two trees consists of levels of vertices of varying degrees $n$ and $m$ and depth $r-1$ (one starting of a root of degree $n-1$ and one starting of a root of degree $m-1$, and the last level consisting of vertices of degree $1$). 
Thus, ${\cal{T}}_{uv}$ is the union of an edge $f=uv$ and two disjoints trees ${\cal{T}}_u$ and ${\cal{T}}_v$, rooted in $f=uv$, where $V({\cal{T}}_{uv})=V({\cal{T}}_u)\cup V({\cal{T}}_v)$ and $E({\cal{T}}_{uv})=E({\cal{T}}_u)\cup E({\cal{T}}_v)\cup f$. A biregular bipartite $(n,m;2r)$-Moore cage exists if and only if
the tree ${\cal{T}}_{uv}$ can be completed into an $(n,m;2r)$-graph by adding 
$n-1$ edges to each leaf of ${\cal{T}}_u$ connecting these leaves 
to the leaves of the tree ${\cal{T}}_v$ in such a way that that each leave of 
${\cal{T}}_v$ ends up being of degree $m$ and no cycles shorter than $2r$ are introduced. 

Using techniques introduced in the study of the $(k;g)$-cage problem (see for example \cite{AGMS07,ABGMV08,ABH10}), we obtain the following extension of 
Theorem~\ref{bb-cages} to other classes of parameters $(n,m;2r)$. 
\begin{theorem}
There exist infinite families of biregular bipartite $(n,m;2r)$-graphs 
with parameters $m=n^2$ for $r= \{4,8\}$, and  $m=n^3$ for $r=6$, whose orders
are
\[ n^{\frac{r}{2}}m^{\frac{r}{2}-1}+m^{\frac{r}{2}}n^{\frac{r}{2}-1}. \]
\end{theorem}

\begin{proof}
Let $G$ be a biregular bipartite Moore $(n+1,m+1;2g)$-cage whose existence is
guaranteed by Theorem~\ref{bb-cages}. As argued above, the Moore graph $G$ contains the Moore tree ${\cal{T}}_{uv}$ with the additional edges connecting the
leaves of the two subtrees ${\cal{T}}_u$ and ${\cal{T}}_v$. Deleting the edges and vertices of ${\cal{T}}_{uv}$ other than the leaves and the edges connecting them 
yields a bipartite biregular graph with the desired parameters. 
\end{proof}

We do not know whether the graphs constructed in the proof of the above theorem are cages or not.



\begin{thebibliography}{10}

\bibitem{AGMS07}
G. Araujo-Pardo, D. Gonz\'alez-Moreno, J.J. Montellano-Ballesteros, O. Serra, 
On upper bounds and connectivity of cages,
{\em Australasian J. Combin.} {\bf 38} (2009) 221-228. 

\bibitem{ABGMV08} 
G. Araujo-Pardo, C. Balbuena, P. Garc\'ia-V\'azquez, X. Marcote, J.C. Valenzuela,
On the order of $(\{r,m\};g)$-cages  of even girth,
{\em Discrete Math.} {\bf{308}} (2008) 2484-2491.

\bibitem{ABH10} 
G. Araujo-Pardo, C. Balbuena, and T. H\'eger, 
Finding small regular graphs of girths 6, 8 and 12 as subgraphs of cages,
{\em Discrete Math.} {\bf 310 (8)} (2010) 1301-1306.

\bibitem{AraExoJaj}
G. Araujo-Pardo, G. Exoo and R. Jajcay, Small biregular graphs of even girth,
{\em Discrete Math.} {\bf 339} (2016) 658-667.


 \bibitem{BobJajPis} M.\ Boben, R.\ Jajcay and T.\ Pisanski, 
Generalized cages,
{\em Electron. J. Combin. } {\bf 22(1)} (2015) \#P1.77.

\bibitem{ColDin} C.J.Colbourn and J.Dinitz (eds.),
{\em The CRC handbook of combinatorial designs. 2nd ed.},
Discrete Mathematics and its Applications, Boca Raton, FL: Chapman \& Hall/CRC  (2007) 984 p.


\bibitem{ErdSachs63} P. Erd\H{o}s and H. Sachs,
Regul\" are Graphen gegebener Taillenweite mit minimaler Knotenzahl,
{\em  Wiss. Z. Uni. Halle (Math. Nat.) } {\bf 12} (1963) 251-257.

\bibitem{ExooJaj08}
G.\ Exoo and R.\ Jajcay,
Dynamic cage survey,
{\em Electron. J. Combin., Dynamic Survey} {\bf 16} (2008).

\bibitem{ExooJaj16}
G. Exoo and R. Jajcay, Biregular cages of odd girth, 
{\em J. Graph Theory} {\bf 81}, No. 1 (2016), 50-56. 




\bibitem{Fil17}
S.\ Filipovski,
On bipartite cages of excess $4$,
{\em Electron. J. Combin.} {\bf 24(1)} (2017) \#P1.40.

 
 \bibitem{FilRamRivJaj19}
S. Filipovski, A. Ramos Rivera and R. Jajcay, On biregular bipartite graphs of small excess,
{\em Discrete Math.} {\bf 342} (2019), 2066-2076. 


\bibitem{FuLaSeUsWol95}
Z.\ Furedi, F.\ Lazebnik, A.\ Seress, V.\ A.\ Ustimenko and A.\ J.\ Woldar, \newblock
Graphs of prescribed girth and bi-degree, \newblock
{\em J. Comb. Theory Ser. B  } {\bf 64} (1995) 228-239.



\bibitem{JaFiJa16}
T. B. Jajcayov\' a, S. Filipovski and R. Jajcay,
Improved lower bounds for the orders of even-girth cages,
{\em Electron. J. Combin. } {\bf 23(3)} (2016) \#P3.55.


\bibitem{VM98}
H. Van Maldeghem,
{\em Generalized Polygons},
Birkh\"auser, Springer, Basel (1998).



\bibitem{YuLi03}
Y. Yuansheng and W. Liang,
The minimum number of vertices with girth $6$ and degree set $D=\{r,m\}$,
{\em Discrete Math.} {\bf  269} (2003) 249-258.

\end{thebibliography}
\end{document}